\documentclass[12pt]{amsart}

\usepackage[utf8]{inputenc}

\usepackage{hyperref}
\usepackage{graphicx}
\usepackage{amssymb}
\usepackage{float}
\usepackage{fullpage}
\usepackage{physics}
\usepackage{comment}
\usepackage{tikz}
\usetikzlibrary{positioning, arrows.meta, shapes.geometric}
\usepackage{mathtools}
\usepackage{bm}
\newtheorem{theorem}{Theorem}[section]
\newtheorem{lemma}[theorem]{Lemma}
\newtheorem{corollary}[theorem]{Corollary}
\newtheorem{proposition}[theorem]{Proposition}

\theoremstyle{definition}

\theoremstyle{remark}
\newtheorem{remark}{Remark}

\numberwithin{equation}{section}

\newcommand{\C}{{\mathbb C}}
\newcommand{\F}{{\mathbb F}}
\newcommand{\R}{{\mathbb R}}

\newcommand{\Z}{{\mathbb Z}}
\newcommand{\1}{{\mathbf 1}}

\newcommand{\eps}{\varepsilon}

\newcommand{\PGL}{\operatorname{PGL}}
\newcommand{\Kl}{\mathrm{Kl}_2}

\renewcommand{\gcd}{}

\renewcommand{\P}{\mathbb{P}}

\newcommand{\n}{\mathrm{n}}

\renewcommand{\n}{\mathrm{n}}

\renewcommand{\pmod}[1]{\,\,(\mathrm{mod}\,{#1})}

\title{The divisor function along arithmetic progressions and binary cubic polynomials}
\author{Lasse Grimmelt}
\address{Isaac Newton Institute for Mathematical Sciences, University of Cambridge,
20 Clarkson Road,
Cambridge,
CB3 0EH
United Kingdom}
\email{lasse.grimmelt@gmail.com}
\author{Jori Merikoski}
\address{Department of Mathematics and Statistics, University of Turku, Yliopistonmäki, 20014 Turku, Finland}
\email{jori.merikoski@gmail.com}
\date{}

\begin{document}
\begin{abstract}
    We prove a new equidistribution estimate for the divisor function in arithmetic progression to moduli that have two small factors. We give two applications. 
    First, we show an asymptotic formula for the divisor function over arithmetic progressions to almost all moduli of exponent $2/3$. Second, we show an asymptotic formula for the divisor function along the nonhomogeneous binary cubic polynomial $X Y^2+1$. 
    \end{abstract}
\maketitle

\section{Introduction}
We expect the divisor function $d(n)$ to be equidistributed over primitive residue classes. To measure this, we define for $(a,q)=1$ and $X> 1$ the well-studied discrepancy
\begin{align*}
    \Delta(X;q,a) :=   \sum_{ \substack{n \leq X \\n \equiv a \pmod{q} }}  d(n)  - \frac{1}{\varphi(q)} \sum_{ \substack{n \leq X \\ (n,q)=1 }}  d(n) .
\end{align*}
For applications it is important to have estimates uniformly in the size of the modulus $q$. Selberg (unpublished) and Hooley \cite{Hooley} independently proved by using the Weil bound for Kloosterman sums that $|\Delta(X;q,a)| \leq X^{o(1)} (q^{1/2}+X^{1/3})$, which is non-trivial in the range $q < X^{2/3-\eps}$.  Fouvry \cite{Fouvry} has shown that $|\Delta(X;q,a)|$ is small on average over $q \in [X^{2/3+\eps},X^{1-\eps}]$. 

These two results leave a gap for moduli in $[X^{2/3-\eps}, X^{2/3+\eps}]$.
Fouvry and Iwaniec \cite{FouvryIwaniec} bound $|\Delta(X;q,a)|$ on average over moduli $q$ with a specific factorization which covers a positive proportion of all moduli in $[X^{2/3-\eps}, X^{2/3+\eps}]$. Cf. also the works of Irving \cite{Irving}, Khan \cite{Khan}, and Liu, Shparlinski, and Zhang \cite{LSZ} for point-wise bounds for more specific types of moduli.

In this paper we prove a new estimate for moduli $q$ that have two small factors, on average over the smaller factor (cf. Theorems \ref{Thm:divisorAP} and \ref{Thm:divisorAPsquares}). This allows us for the first time to obtain a non-trivial bound for $|\Delta(X;q,a)|$ for almost all moduli of size $[X^{2/3-\eps}, X^{2/3+\eps}]$. 
\begin{theorem} \label{thm:APexceptional}
There exists a constant $B> 1$ such that uniformly for $\eps >0$, $a \in \Z$, and $Q < X^{2/3+\eps}$, we have
    \begin{align*}
     \Delta(X;q,a)  \ll \frac{X^{1- \eps} (\log X)^{B}}{q}
    \end{align*} 
    for all but $O( \sqrt{\eps} Q) $ moduli $q\in (Q,2Q]$ with $(a,q)=1$.
\end{theorem}
In particular, for any $A > 0$ and $Q \leq X^{2/3} (\log X)^A$ we have
   \begin{align*}
     \Delta(X;q,a) \ll \frac{X }{q (\log X)^{A}}
    \end{align*}
for all but $\ll_A \frac{ (\log \log X)^{1/2}}{ (\log X)^{1/2}} Q$ moduli $q \in (Q,2Q]$ with $(a,q)=1$.

As a second application, we consider the divisor function over binary cubic polynomials. Previously an asymptotic formula has been shown by Greaves \cite{greaves} for binary cubic forms (cf.  also the work of Browning \cite{browning}). We are for the first time able to address a nonhomogeneous binary cubic polynomial.
\begin{theorem} \label{thm:binarycubic}
  For $C(X,Y)=X Y^2+1$  we have
    \begin{align*}
        \sum_{n_1,n_2 \leq X} d(C(n_1,n_2)) = \Bigl(1+ O\Bigl( \Bigl(\frac{\log \log X}{ \log X} \Bigr)^{1/2} \Bigr)\Bigr) \frac{3}{\zeta(2)} X^2 \log X. 
    \end{align*}
\end{theorem}

\subsection{Factorable estimates}
Our main technical results are factorable equidistribution estimates for the divisor function that are proved in Section \ref{sec:factorableproof}. For the statement we define a weighted version of the discrepancy
\begin{align*}
     \Delta^\psi(X;q,a) :=   \sum_{ \substack{n}} \psi(n/X) d(n) \Bigl(\1_{n \equiv a \pmod{q}} - \frac{\1_{(n,q)=1}}{\varphi(q)} \Bigr).
\end{align*}
\begin{theorem}\label{Thm:divisorAP}
    Let $r,s \geq 1$ be coprime with $s$ cube-free and let $U,X>1$. Let $\delta \in (0,1)$ and let $\psi:\R\to\C$ be a smooth function supported on $[1,2]$ and satisfying for all $j \geq 0$ the derivative bound $\|\psi^{(j)}\|_\infty \ll_j  \delta^{-j}$. Then uniformly for all $a \in \Z$ with $(a,rs)=1$ we have
    \begin{align*}
        & \frac{1}{U}\sum_{\substack{u \sim  U \\ u \, \text{square-free} \\ (u,ars)=1 \\} } | \Delta^\psi(X;rsu,a)|  \ll   \delta^{-2}d(rs)^{O(1)}(\log X)^{O(1)}  \mathcal{L}_{r,s}(U,X),  \\
       & \mathcal{L}_{r,s}(U,X) =     s^{1/8}U^{1/4}X^{1/4} + r^{1/4}    X^{1/4} + (rs)^{1/2}  +  r^{1/2} s^{3/8} U^{1/4} .
    \end{align*}
\end{theorem}
Here, as usual the summation condition $u \sim U$ means $U<u\leq 2U$. We take the variables to satisfy the hierarchy $\log s \gg \log r \gg \log u.$ More precisely, for   $rsU = X^{2/3+o(1)}$ Theorem \ref{Thm:divisorAP} gives a power saving in the range
\begin{align} \label{eq:powerranges}
    s^3 U^6 < X^{2-\eps},   \quad r \leq X^{1/3-\eps},  \quad U  > X^{\eps}, \quad\text{and} \quad r^{12} s^9 U^6 < X^{8-\eps}.
\end{align}
Importantly, for this to be non-trivial the factors $r$ and $u$ can be made small. Generic moduli $q= X^{2/3+o(1)}$ have such factors and this allows us to obtain Theorem \ref{thm:APexceptional}.

We show a similar factorable estimate with an average over perfect squares of square-free integers from which Theorem \ref{thm:binarycubic} will follow by restricting to $n_2$ with a suitable factorization. The sparseness of squares makes the upper bound slightly worse in the third and fourth term but it still gives a power saving in the same range \eqref{eq:powerranges} for   $rsU = X^{2/3+o(1)}$.
\begin{theorem}\label{Thm:divisorAPsquares}
    Let $r,s \geq 1$ be coprime with $s$ cube-free  and let $U,X > 1$. Let $\psi:\R\to\C$ be a smooth function supported on $[1,2]$ and satisfying for all $j \geq 0$ the derivative bound $\|\psi^{(j)}\|_\infty \ll_j  \delta^{-j}$. Then uniformly for all $a \in \Z$ with $(a,rs)=1$ we have
    \begin{align*}
       & \frac{1}{U^{1/2}}\sum_{\substack{u^2 \sim  U \\ u \, \text{square-free} \\ (u,ars)=1 \\} } |  \Delta^\psi(X;rsu^2,a)|  
        \ll \delta^{-2}d(rs)^{O(1)} (\log X)^{O(1)}    \mathcal{M}_{r,s}(U,X), \\
     &   \mathcal{M}_{r,s}(U,X) =     s^{1/8}U^{1/4}X^{1/4} + r^{1/4}   X^{1/4} + (rs)^{1/2} U^{1/4} +  r^{1/2} s^{3/8} U^{1/4+o(1)}.
    \end{align*}
\end{theorem}

\subsection{Remarks}
\begin{enumerate}
    \item The proofs of Theorems \ref{Thm:divisorAP} and \ref{Thm:divisorAPsquares} generalize to the divisor function replaced by Hecke eigenvalues of GL(2) holomorphic or Maass cuspidal newforms. Indeed, we only need two properties of the divisor function, namely, the Vorono\"\i\, summation formula and an $\ell^2$-norm bound.
    \item It seems likely that with more work one can establish Theorems \ref{Thm:divisorAP} and \ref{Thm:divisorAPsquares} with full uniformity in the residue class $a$, that is, with a maximum over $a$ as in the classical Bombieri-Vinogradov theorem. As a corollary, this would allow an extension  of Theorem \ref{thm:binarycubic} to more general binary cubic polynomials $C(X,Y)=X Y^2 +P(Y)$ with $P$ an arbitrary cubic polynomial. As it is, the proof can be generalized to $XY^2+a.$
    \item The constant $B$ in Theorem \ref{thm:APexceptional} is effective and depends linearly on the implied constants in Lemmas \ref{le:completeKL} and  \ref{le:primesquares}. All other implied constants are also effective.
\end{enumerate}

\subsection{Sketch of the proof}
Consider Theorem \ref{Thm:divisorAP} for
moduli $q=rsu \sim Q=X^{2/3}$. Then by Vorono\"\i\, summation we morally have
\begin{align*}
  \Delta(X;q,a) \approx \sum_{n \leq \sqrt{Q}} d(n) \Kl(an;q),
\end{align*}
where $\Kl(an;q)$ denotes the normalized Kloosterman sums \eqref{eq:kldef}. The point-wise Estermann-Weil bound \eqref{eq:Weil} barely fails for moduli of size $Q=X^{2/3}$ and so our task is to show additional cancellation. The average over a small factor $u \sim U$ allows us to apply Cauchy-Schwarz to smoothen the variable $n$. This reduces the problem to showing non-trivial cancellation in
\begin{align*}
    \sum_{u_1,u_2 \sim U} \bigg| \sum_{n \leq \sqrt{Q}}  \Kl(an;rsu_1) \Kl(an;rsu_2) \bigg|.
\end{align*}
Note that the size of the modulus has increased from $Q$ to $QU$ and so the sum over $n$ is barely too short for the Pólya-Vinogradov method. The second small factor $r$ allows us to use Heath-Brown's $q$-van der Corput method and on the off-diagonal $u_1 \neq u_2$ we morally get
\begin{align*}
    \sum_{n \leq \sqrt{Q}}  \Kl(an;rsu_1) \Kl(an;rsu_2) \ll Q^{1/4} (s^{1/4}U^{1/2} +r^{1/2}).
\end{align*}
This gives a non-trivial saving as long as $r$ is larger than $U$ and smaller than $\sqrt{Q}$. For the proof we need a Deligne type bound for the correlation of four Kloosterman sums (Lemma \ref{le:completeKL}) due to Fouvry, Kowalski, and Michel \cite{FKMproducts}. For the proof of Theorem \ref{Thm:divisorAPsquares} we prove an analogous bound with square moduli (Lemma \ref{le:primesquares}).
\subsection{Acknowledgements}
The authors are grateful to James Maynard for helpful discussions. This project has
received funding from the European Research Council (ERC) under the European Union's Horizon research and innovation programme (grant agreement No. 851318 both authors, No. 101162746 first author). The second author was supported by a grant from the Magnus Ehrnrooth foundation.

\section{Lemmas on exponential sums}
\subsection{Complete sums modulo primes}
We begin by importing a result of Fouvry, Kowalski, and  Michel \cite{FKMproducts} on correlations of trace functions, which builds on bounds of Deligne and Katz coming from algebraic geometry. Let $p$ be a prime. The group $\PGL_2(\F_p)$ acts on $\F_p$ (or more precisely on $\P_{\F_p}$) via 
\begin{align*}
    \gamma = \mqty(a&b \\ c  & d): x \mapsto \gamma.x := \frac{ax+b}{cx+d}. 
\end{align*}
For $q \geq 1$ and $a \in \Z$ we denote the usual (normalised) Kloosterman sum by
\begin{align} \label{eq:kldef}
    \Kl(a;q) := \frac{1}{\sqrt{q}}\sum_{x  \, (q)^\ast} e_q(ax+ \overline{x}).
\end{align}
Recall that it is real valued and that for any $\gcd(r,s)=1$ we have the twist multiplicative relation 
\begin{align}\label{eq:twistmultipl}
     \Kl(a;rs) =   \Kl(\overline{r}^2 a;s)  \Kl(\overline{s}^2 a;r).
\end{align}
We also recall the Estermann-Weil bound \cite{Estermann} for all integers $c \geq 1$
\begin{align} \label{eq:Weil}
    | \Kl(a;c)| \leq d(c) \gcd(a,c)^{1/2}.
\end{align}

The following lemma is \cite[Corollary 3.3]{FKMproducts}. We only need the special cases $k=2$ and $k=4$ and $\gamma_i.x$ being affine functions $ax+b$. 
\begin{lemma} \label{le:completeKL}
 Let $k \geq 1$ be an integer.   Let $p$ be a prime, $h \in \F_p$, and let $\gamma_1,\dots,\gamma_k \in \PGL_2(\F_p)$. Suppose that either $h \neq 0$ or that there exists some $\gamma \in \PGL_2(\F_p)$ with
    \begin{align*}
        \# \{i \leq k: \gamma_i=\gamma\} \equiv 1 \pmod{2}.
    \end{align*}
    Then
    \begin{align*}
      \sum_{x \, (p)} e_p(xh)   \prod_{i=1}^k \Kl(\gamma_i.x;p) \ll_k p^{1/2},
    \end{align*}
    where the sum runs over $x \in \F_p$ such that $\gamma_i.x$ is defined for all $i \leq k$.
\end{lemma}
\subsection{Complete sums modulo prime squares}
For Theorems \ref{thm:binarycubic} and \ref{Thm:divisorAPsquares}  we require a similar result for prime square moduli. We only state and prove it in the special case of affine functions necessary for us, but it seems likely that a full analogue of Lemma \ref{le:completeKL} also holds.

\begin{lemma} \label{le:primesquares}
Let $a_1, a_2, b_1, b_2,h \in \Z$ and $p$ a prime with $\gcd(a_1a_2,p)=1$. Then we have
\begin{align} \label{eq:twoKLp2}
     \sum_{x\, (p^2)} e_{p^2}(xh)    \Kl(a_1(x+b_1);p^2) \Kl(a_1(x+b_2);p^2) \ll p \gcd(h,b_1-b_2,p)
\end{align}
and
\begin{align}\begin{split}
     \sum_{x\, (p^2)} e_{p^2}(xh)   \prod_{i,j=1}^2  \Kl(a_i(x+b_j);p^2)  
  \ll p \bigl(   \gcd (h,a_1^2-a_2^2,p) + \gcd(h,b_1-b_2,p)\bigr).\label{eq:fourKLp2}
\end{split}
\end{align}
\end{lemma}
\begin{proof}
The proof of \eqref{eq:twoKLp2} is similar but easier so we only consider \eqref{eq:fourKLp2}. By writing $x=y+zp$ for $y,z \in \Z/p\Z$ and using $\overline{y+pz} \equiv \overline{y}-pz \overline{y}^2 \pmod{p^2}$, we have
\begin{align*}
    \Kl(a;p^2) = \frac{1}{p} \sum_{y \, (p)^\ast} \sum_{z \,(p)} e_{p^2}(ay+\overline{y}) e_p(z(a-\overline{y}^2)) = \sum_{\substack{y \, (p)^\ast  \\ \overline{y}^2 \equiv a \pmod{p}}}   e_{p^2}(ay+\overline{y}).
\end{align*}
Plugging this in four times and writing again  $x=y+zp$, we get
\begin{align*}
    & \bigg|  \sum_{x\, (p^2)} e_{p^2}(xh)    \prod_{i,j=1}^2  \Kl(a_i(x+b_j);p^2)  \bigg|  \leq \sum_{\substack{y,y_1,y_2,y_3,y_4 \pmod{p} \\  \overline{y_i}^2 \equiv a_i y +b_i \pmod{p} }}\bigg| \sum_{z \pmod{p}} e_p(zh) \prod_{i,j=1}^2 e_{p}( a_i z y_{i+j} ) \bigg| \\
&\leq p \, \# \Bigl\{ (y,y_1,y_2,y_3,y_4) \in (\Z/p\Z)^5:  \overline{y_{i+j}}^2 \equiv a_i (y +b_j) \pmod{p}  \forall i,j \in \{1, 2\}  \\
& \hspace{175pt}  h+a_1 y_1+a_2 y_2+a_1 y_3+a_2 y_4\equiv 0  \pmod{p} \Bigr \}.
\end{align*}
It then suffices to show that the number of solutions $(y,y_1,y_2,y_3,y_4)$ is bounded by
\begin{align*}
 \ll   \gcd (h,a_1^2-a_2^2,p) + \gcd(h,b_1-b_2,p).
\end{align*}
Solving for $y$, we consider a system of four equations
\begin{align*}
 a_1 y_1^2 = a_2 y_2^2, \quad \quad
  a_1 y_3^2&= a_2 y_4^2,\\
    a_1(y_3^2-y_1^2) + (b_2-b_1)y_1^2 y_3^2 &= 0,\\
       h+a_1 y_1+ a_2 y_2 + a_1 y_3 + a_2 y_4 &= 0.
\end{align*}
Fixing roots $\omega_1,\omega_2 \in \Z/p\Z$ to the equation $\omega^2 = a_1/a_2$ we get  $ y_2 =  \omega_1 y_1 $ and $
 y_4 =  \omega_2 y_3 $. Substituting for $y_2$ and $y_4$, it suffices to bound the number of solutions to
\begin{align*}
  a_1(y_3^2-y_1^2) + (b_2-b_1)y_1^2 y_3^2 &= 0,\\
 h+(a_1+ \omega_1 a_2) y_1 + (a_1 + \omega_2 a_2) y_3 &= 0. 
\end{align*}
If $(a_1 + \omega_1 a_2)= (a_1 + \omega_2 a_2)=0$, then there are no solutions unless $h=0$ and  in this case the number of solutions is $O(\gcd(h,a_1^2-a_2^2,p))$. Suppose then by symmetry that $(a_1 + \omega_1 a_2) \neq 0$. We can solve $y_1=\frac{-h- (a_1 + \omega_2 a_2) y_3}{(a_1 +\omega_1 a_2)}$ and plug it in to get
\begin{align*}
    a_1\left(y_3^2-\Bigl(\frac{h+ (a_1  +\omega_2 a_2) y_3}{(a_1+ \omega_1 a_2)}\Bigl)^2\right) + (b_2-b_1)\Bigl(\frac{h+ (a_1 + \omega_2 a_2) y_3}{(a_1 + \omega_1 a_2)}\Bigr)^2 y_3^2 &= 0.
\end{align*}
If $(a_1 + \omega_2 a_2) \neq 0$ and $b_1\neq b_2$, then the congruence for $y_3$ is non-trivial by considering the leading term  $ y_3^4 (b_1-b_2)\frac{(a_1 + \omega_2 a_2)^2 }{(a_1+ \omega_1 a_2)^2}$ and we get $O(1)$ solutions. If $b_1=b_2$, then we get
\begin{align*}
      a_1\left(y_3^2-\Bigl(\frac{h+ (a_1  +\omega_2 a_2) y_3}{(a_1+ \omega_1 a_2)}\Bigl)^2\right) = 0
\end{align*}
which is non-trivial as long as $h \neq 0$. Thus, in this case we get at most $\ll \gcd(h,b_1-b_2,p)$ solutions. Finally, if $(a_1 + \omega_2 a_2) = 0$, then we get
\begin{align*}
    a_1\left(y_3^2+\Bigl(\frac{h}{(a_1+\omega_1 a_2)}\Bigl)^2\right)+(b_2-b_1)\Bigl(\frac{h}{(a_1+ \omega_1 a_2)}\Bigl)^2 y_3^2 = 0
\end{align*}
If $h \neq 0$, then either there are no solutions (if the leading coefficients cancel) or there are $O(1)$ solutions because the polynomial is non-trivial. For $h=0$ the equation is $a_1 y_3^2 = 0$ and there is only one solution.
\end{proof}

\subsection{Incomplete sums: Pólya-Vinogradov bound}
In this section we use the Pólya-Vinogradov method of completing the sum to get the following bound (cf. \cite[Sections 3 and 6]{Polymath} for similar arguments).
\begin{lemma}\label{lem:PV}
Let $a,s,u_1,u_2 \geq 1$ be pair-wise coprime integers with $s,u_1,u_2$ be cube-free. Let $N > 1$ and let $\psi_N:\R \to [0,1]$ be a smooth function supported on $[-N,N]$ satisfying $\psi_N^{(j)} \ll_j N^{-j}$.
 We have
    \begin{align*}
     &\sum_{n}  \psi_N(n) \prod_{i,j=1}^2 \Kl(a(n+b_j) ;su_i) \ll  d (su_1u_2)^{O(1)}   (su_1u_2)^{1/2} \mathcal{P}_{s,u_1,u_2}(N),   \\
       &\mathcal{P}_{s,u_1,u_2}(N) = 1 + \frac{N}{s u_1 u_2} \prod_{p^k|| u_1 u_2}\gcd(b_1-b_2,p)^{k/2}   \prod_{p^k|| s} \big( \gcd(u_1^2-u_2^2,p)^{k/2 } + \gcd(b_1-b_2,p)^{k/2} \big).    \end{align*}
\end{lemma}

\begin{proof}
We complete the sum over $n$ with Poisson summation
\begin{align*}
      \sum_{n}  \psi_N(n) \prod_{i,j=1}^2 \Kl(a(n+b_j) ;su_i)
        = \frac{1}{s u_1 u_2}\sum_{h} \widehat{\psi_N}\Bigl(\frac{h}{s u_1 u_2}\Bigr)  S(h;s,u_1,u_2),
\end{align*}
where
\begin{align*}
    S(h;s,u_1,u_2) = \sum_{x \, (s u_1 u_2 )} e_{s  u_1,u_2 }(hx) \prod_{i,j=1}^2 \Kl(a(x+b_j) ;su_i).
\end{align*}
By \eqref{eq:twistmultipl} and the Chinese remainder theorem we can write
\begin{align*}
    S(h;s,u_1,u_2) = \prod_{p^k|| s}  T(h;p^k) \prod_{p^k|| u_1}  U(h,u_1;p^k) \prod_{p^k|| u_2} U(h,u_2;p^k),
\end{align*}
where for $q=p^k||s$ and $c=su_1u_2$
\begin{align*}
    T(h;q) :=   \sum_{x \, (q)} e_{q}(hx \overline{(c/q)} ) \prod_{i,j=1}^2  \Kl(\overline{u_i}^2 \overline{(s/q)}^2a (x +b_j) ;q) 
\end{align*}
and for $q=p^k||u_i$
\begin{align*}
    U(h,u_i;q) :=& \sum_{x \, (q)} e_{q}(hx \overline{(c/q)} )  \Kl( (\overline{s u_i/q })^2a (x +b_1) ;q)  \Kl((\overline{su_i/q})^2a( x+b_2);q).
\end{align*}
 For $p^k||s$ we get from Lemma \ref{le:completeKL} if $k=1$ or \eqref{eq:fourKLp2} of Lemma \ref{le:primesquares} if $k=2$
\begin{align*}
     T(h;p) \ll \gcd(h,u_1^2-u_2^2,p)^{k/2} p^{k/2} + \gcd(h,b_1-b_2,p)^{k/2} p^{k/2}.
\end{align*}
For $p||u_i$ we get from Lemma \ref{le:completeKL} if $k=1$ or \eqref{eq:twoKLp2} of Lemma \ref{le:primesquares} if $k=2$
\begin{align*}
     U(h,u_i;p^k) \ll  \gcd(h,b_1-b_2,p)^{k/2} p^{k/2}.
\end{align*}
Summing over $h$ gives the result.
\end{proof}

\subsection{Incomplete sums: $q$-van der Corput bound}
We require the following estimate that follows by an application of Heath-Brown's $q$-analogue of the van der Corput $A$-process \cite{HBLfunctions,Polymath} and Lemma \ref{lem:PV}.
\begin{lemma}\label{lem:qvdc}
 Let $a,r,s,u_1,u_2 \geq 1$ be pair-wise coprime integers with $s,u_1,u_2$ cube-free and let $c|r$. Let $N > 1$ and let $\psi_N:\R \to [0,1]$ be a smooth function supported on $[1/2,N]$ satisfying $\psi_N^{(j)} \ll_j N^{-j}$. Then
   \begin{align*}
       \sum_{(n,r)=c} \psi_N(n) &\Kl(an;rsu_1) \Kl(an;rsu_2)  \\
       &\ll  d(r s u_1 u_2)^{O(1)}  (c N)^{1/2}\Bigl(  (su_1u_2)^{1/4} \Bigl(1 + \frac{N\prod_{p^k|| s} \gcd(u_1^2-u_2^2,p)^{k/2 } }{su_1u_2}\Bigr)^{1/2}  + r^{1/2}  \Bigr).
   \end{align*}
\end{lemma}

\begin{proof}
    We note that for any $k \in \Z$ by the twist multiplicative relation \eqref{eq:twistmultipl} and periodicity
    \begin{align*}
        &\Kl(a(n+kr);rsu_1) \Kl(a(n+kr);rsu_2) \\
        = &\Kl(a\overline{su_1}^2n;r) \Kl(a\overline{su_2}^2 n;r)\Kl(a\overline{r}^2(n+kr);su_1) \Kl(a\overline{r}^2(n+kr);su_2).
    \end{align*}
    Therefore, denoting $K = \lfloor N/r \rfloor $, we have by \eqref{eq:Weil}
    \begin{align*}
          &\bigg|\sum_{(n,r)=c} \psi_N(n) \Kl(an;rsu_1) \Kl(an;rsu_2)\bigg| \\
          &= \bigg| \frac{1}{K} \sum_{(n,r)=c}  \sum_{\substack{k \leq K }} \psi_N(n+kr) \Kl(n+kr;rsu_1) \Kl(n+kr;rsu_2)\bigg| \\
       &   \leq   \frac{c}{K} \sum_{\substack{n \ll N\\ (n,r)=c}}  \bigg| \sum_{\substack{k \leq K }}  \psi_N(n+kr) \Kl(a\overline{r}^2(n+kr);su_1) \Kl(a\overline{r}^2(n+kr);su_2) \bigg|.
    \end{align*}

    Applying Cauchy-Schwarz on $n$ we get
    \begin{align*}
        \sum_{(n,r)=c} \psi_N(n) \Kl(an;rsu_1) \Kl(an;rsu_2) \ll \frac{(cN)^{1/2}}{K}  \bigg( \sum_{n} \bigg| \sum_{\substack{k \leq K }} \bigg|^2 \bigg)^{1/2}
    \end{align*}
   and by the choice $K = \lfloor N/r \rfloor $ it suffices to show that
   \begin{align} \label{eq:claimafterCS}
       \sum_n \bigg| \sum_k \bigg|^2 \ll d(rsu_1u_2)^{O(1)} \Bigl( K^2  (su_1u_2)^{1/2} (1 + \frac{N\gcd(u_1^2-u_2^2,s)^{1/2 } }{su_1u_2})  + KN  \Bigr).
   \end{align}
By expanding the square, the sum we need to estimate becomes
    \begin{align*}
        \sum_{k_1,k_2 \leq K} \sum_{\substack{n \\ }} \psi_N(n+k_1r) \psi_N(n+k_2r) \prod_{i,j=1}^2\Kl(a\overline{r}^2(n+k_jr);su_i).
    \end{align*}
We apply Lemma \ref{lem:PV} to bound this by
\begin{align*}
    \ll &d(rsu_1u_2)^{O(1)}   (su_1u_2)^{1/2}   \\
    &\times \sum_{k_1,k_2 \leq K} \Bigl(1 + \frac{N}{ s u_1 u_2} \prod_{p^k|| u_1 u_2}\gcd(k_1-k_2,p)^{k/2}  \prod_{p^k|| s} \big( \gcd(u_1^2-u_2^2,p)^{k/2 } + \gcd(k_1-k_2,p)^{k/2} \Bigr).
\end{align*}
For $k'=k_1-k_2 \neq 0$ we crudely bound $\gcd(u_1^2-u_2^2,p)^{k/2 } + \gcd(k',p)^{k/2}  \leq \gcd(u_1^2-u_2^2,p)^{k/2} \gcd(k',p)^{k/2} \leq \gcd(u_1^2-u_2^2,p)^{k/2} \gcd(k',p^k)$ and obtain the estimate
\begin{align*}
    \sum_{k_1,k_2 \leq K}  &\leq K^2 + \frac{N}{su_1u_2} \bigg(K (su_1u_2)^{1/2}  +  K \sum_{0<|k'| \leq 2K} \gcd(k',su_1u_2)  \prod_{p^k|| s} \gcd(u_1^2-u_2^2,p)^{k/2 }\bigg)   \\
   & \leq  K^2 + \frac{N}{su_1u_2} (K (su_1u_2)^{1/2}  + d(su_1u_2) K^2 \prod_{p^k|| s} \gcd(u_1^2-u_2^2,p)^{k/2 } ) .
\end{align*}
This  gives the claim \eqref{eq:claimafterCS}.
\end{proof}

\section{Certain bilinear sums of Kloosterman sums}
In this section we apply Lemma \ref{lem:qvdc} to give the following bound for a bilinear sum of Kloosterman sums.
\begin{proposition} \label{le:klforms}
     Let $r,s \geq 1$ be coprime with $s$ cube-free and let $U> 1$. Let $\lambda_n,\gamma_u$ be complex coefficients with $\gamma_u$ supported on a subset of cube-free integers $\mathcal{U}$ such that for all $u_1,u_2 \in \mathcal{U}$ with $u_0=\gcd(u_1,u_2)$ we have $\gcd(u_1/u_0, u_0) = 1$. Denote $\widetilde{\gamma}_u = \gamma_u d(u)^{O(1)}$ and $\lambda^{(c)}_n := \lambda_n c^{1/4} \mathbf{1}_{c|n}$. Then uniformly in $a$ with $(a,rs)=1$ we have
    \begin{align*}
      &  \sum_{\substack{u \sim  U  \\ (u,ars)=1} } \gamma_u  \sum_{1 \leq n \leq N} \lambda_n  \Kl(an,rsu)  \ll d(rs)^{O(1)} \sum_{c|r}\|\lambda^{(c)}\|_2 \, \mathcal{K}_{r,s,\gamma}(U,N), \\ &
      \mathcal{K}_{r,s,\gamma}(U,N) =  N^{1/4} \|\widetilde{\gamma} \|_1 ( s^{1/8} U^{1/4}  +     r^{1/4} )  + N^{1/2} \Bigl( \sum_{u_0 \leq U}    \sum_{\substack{v_1,v_2 \sim  U/u_0 \\ (v_1,v_2)=1 } } |\widetilde{\gamma}_{u_0v_1}   \widetilde{\gamma}_{u_0v_1}|  \frac{\prod_{p^k|| s u_0} \gcd(v_1^2-v_2^2,p)^{k/4 }  }{(su_0v_1v_2)^{1/4} } \Bigr)^{1/2}. 
    \end{align*}
\end{proposition}
\begin{proof}
    We sort according to $c=(n,r)$ to write
    \begin{align*}
        \sum_{\substack{u \sim  U  \\ (u,ars)=1} } \gamma_u  \sum_{\substack{1 \leq n \leq N }}  \lambda_n  \Kl(an,rsu)  = \sum_{c|r} \sum_{\substack{u \sim  U  \\ (u,ars)=1} } \gamma_u  \sum_{\substack{1 \leq n \leq N \\ (n,r)=c}}   c^{-1/4}\lambda^{(c)}_n  \Kl(an,rsu) .
    \end{align*}
    By applying Cauchy-Schwarz on $n$ we get
    \begin{align*}
          \sum_{\substack{u \sim  U  \\ (u,ars)=1} } \gamma_u  \sum_{\substack{1 \leq n \leq N \\ (n,r)=c}} \lambda_n  \Kl(an,rsu)  \leq \sum_{c|r}\|\lambda^{(c)}\|_2 \bigg( \sum_{\substack{(n,r)=c}} \frac{\psi_N(n)}{\sqrt{c}} \bigg| \sum_{\substack{u \sim  U  \\ (u,ars)=1} } \gamma_u  \Kl(an,rsu) \bigg|^2 \bigg)^{1/2}.
    \end{align*}
Expanding the square and making the change of variables $u_j = u_0 v_j$ with $u_0=(u_1,u_2)$  we get
\begin{align*}
  \sum_{\substack{(n,r)=c}}  \bigg| \sum_{\substack{u \sim  U  } } \bigg|^2  = \sum_{u_0 \leq U} \sum_{\substack{v_1,v_2 \sim U/u_0 \\ (v_1,v_2)=1 \\ (u_0v_1v_2,ars)=1}} \gamma_{u_0v_1}   \gamma_{u_0v_1}  \sum_{\substack{(n,r)=c}} \frac{\psi_N(n)}{\sqrt{c}}  \Kl(an,rsu_0 v_1) \Kl(an,rsu_0v_2).
\end{align*} 
Using Lemma \ref{lem:qvdc} with $ s \mapsto s u_0$, making use of the support condition of $\gamma$ to guarantee $(v_1v_2,u_0)=1$, this is bounded by
\begin{align*}
  \ll &d(rs)^{O(1)} \Bigl(N^{1/2} \|\widetilde{\gamma} \|_1^2 (  s^{1/4} U^{1/2} + r^{1/2}) + N \sum_{u_0 \leq U}    \sum_{\substack{v_1,v_2 \sim  V/u_0 \\ (v_1,v_2)=1 } }  |\widetilde{\gamma}_{u_0v_1}   \widetilde{\gamma}_{u_0v_2}| \frac{  \prod_{p^k|| s u_0} \gcd(v_1^2-v_2^2,p)^{k/4 }}{(su_0v_1v_2)^{1/4} } \Bigr).
\end{align*}
\end{proof}

As an immediate Corollary, we obtain for the case of bounded $\gamma$ on square-free coefficients the following.
\begin{corollary} \label{cor:sqfree}
      Let $r,s \geq 1$ be coprime with $s$ cube-free and let $U,N> 1$. Let $\lambda_n,\gamma_u$ be divisor bounded complex coefficients with $\gamma_u$ supported on square-free integers. Then uniformly in $a$ with $(a,rs)=1$ we have
    \begin{align*}
         & \sum_{\substack{u \sim  U  \\ (u,ars)=1} } \gamma_u  \sum_{1 \leq n \leq N} \lambda_n  \Kl(an,rsu)  \\
         \ll&  d(rs)^{O(1)}(\log 2N U)^{O(1)}N^{1/2} \bigg( N^{1/4} ( s^{1/8} U^{5/4}  +     r^{1/4} U  )  
         +  N^{1/2} U^{1/2} + \frac{N^{1/2} U^{3/4}}{s^{1/8}}\bigg).
    \end{align*}
\end{corollary}
\begin{proof}
    This follows by Proposition \ref{le:klforms} after using Cauchy-Schwarz and factoring $v_1^2-v_2^2 = (v_1-v_2)(v_1+v_2)=v_3v_4$ to get for any $V>1$ and $q=su_0$ cube-free
\begin{align*}
   \sum_{\substack{v_1,v_2 \sim  V \\ (v_1,v_2)=1 } }  d(v_1v_2)^{O(1)}  &\prod_{p^k|| q}\gcd(v_1^2-v_2^2,p)^{k/4}  \ll
   \sum_{\substack{v_1,v_2 \sim  V \\ (v_1,v_2)=1 } }  d(v_1v_2)^{O(1)}  \gcd(v_1^2-v_2^2,q)^{1/2}  \\
   &\ll V (\log V)^{O(1)} \bigg(  \sum_{\substack{v_1,v_2 \sim  V \\ (v_1,v_2)=1 } }  \gcd(v_1^2-v_2^2,q) \bigg)^{1/2} \\
  &  \ll V (\log V)^{O(1)} \bigg(   \sum_{\substack{0 < |v_3|,|v_4| \leq  2V  } }   \gcd(v_3,q) \gcd(v_4,q) \bigg)^{1/2}  \ll d(q) V^2 (\log V)^{O(1)}.
\end{align*}
Note that $v_1\neq \pm v_2$ except when $V=1$ and $v_1=v_2=1$ which contributes the term $N^{1/2} U^{1/2}$ in the corollary.
\end{proof}

For $u$ restricted to squares of square-free integers we get the following estimate.
\begin{corollary} \label{cor:squares}
      Let $r,s \geq 1$ be pair-wise coprime with $s$ cube-free and let $N,U> 1$. Let $\lambda_n,\gamma_u$ be divisor-bounded complex coefficients with $\gamma_u$ supported on square-free integers. Then uniformly in $a$ with $(a,rs)=1$ we have
    \begin{align*}
         &\sum_{\substack{u^2 \sim  U  \\ (u,ars)=1} } \gamma_u  \sum_{1 \leq n \leq N} \lambda_n  \Kl(an,rsu^2)  \\
        & \ll d(rs)^{O(1)}(\log 2NU)^{O(1)}N^{1/2} \bigg( N^{1/4} ( s^{1/8} U^{3/4}  +     r^{1/4} U^{1/2}  )  
         +  N^{1/2} U^{1/4} + \frac{N^{1/2} U^{1/4+o(1)} }{s^{1/8}}\bigg).
    \end{align*}
\end{corollary}
\begin{proof}
    This follows by Proposition \ref{le:klforms}, using $d(v_1v_2) \leq U^{o(1)}$, and for $V>1$ using the factorization $v_1^4-v_2^4=(v_1^2+v_2^2)(v_1-v_2)(v_1+v_2)$ to get the bound
\begin{align*}
 \sum_{\substack{v_1^2,v_2^2 \sim V \\ (v_1,v_2)=1}}  \prod_{p^k|| q} \gcd(v_1^4-v_2^4,p)^{k/4 }  & \leq \sum_{\substack{v_1^2,v_2^2 \sim V \\ (v_1,v_2)=1}} \gcd(v_1^4-v_2^4,q)^{1/2}  \\
 &\leq  \sum_{\substack{v_1^2,v_2^2 \sim V \\ (v_1,v_2)=1}} \gcd(v_1^2+v_2^2,q)^{1/2}  \gcd((v_1-v_2)(v_1+v_2),q)^{1/2}   \\
    &\leq  \sum_{\substack{v_1^2,v_2^2 \sim V \\ (v_1,v_2)=1}} (\gcd(v_1^2+v_2^2,q)   +\gcd((v_1-v_2)(v_1+v_2),q))  \ll d(q)^2 V.
\end{align*}
\end{proof}

\begin{remark}
    Similar estimates as in this section can be shown for many other trace functions in place of $\Kl$. In particular, in the situation of \cite[Theorems 2.1 and 2.2]{KStrace}, one can show bounds for almost all integer moduli $q$  in ranges of $N$ slightly below $q^{1/2}$.
\end{remark}
\section{Proofs of the factorable estimates}
\label{sec:factorableproof}

\subsection{Proof of Theorem \ref{Thm:divisorAP}}
Recall that $\|\psi^{(j)}\|_\infty \ll_j  \delta^{-j}$. For any $T > 1$ we denote $N_T = \delta^{-2}  T^2 /X  $. By Vorono\"\i\, summation similarly as in \cite[Section 3]{FouvryIwaniec} and splitting $t$  into dyadic ranges $t\sim T$, we have
  \begin{align*}
        | \Delta^\psi(X;rsu,a)| 
          \ll & \max_{\substack{T \ll rs U \\ N > N_T}} \frac{X(\log X)^2}{rsu T^{1/2} (N/N_T)^2}   \sum_{\substack{t|rsu \\t \sim T}}  \bigg| \sum_{1 \leq n \leq N} \lambda_n  \Kl(an,t)  \bigg|
    \end{align*}
for some $\lambda_n$ with $|\lambda_n| \leq d(n)$ which do not depend on the moduli.  Writing $t=r_0 s_0  w$ with $u=vw$, we have $w\asymp W:=T/r_0s_0$.  Summing trivially over $v \asymp Ur_0s_0/T$, we get
\begin{align*}
    \sum_{\substack{u \sim  U \\ u \, \text{square-free} \\ (u,ars)=1 \\} } | \Delta^\psi(X;rsu,a)| 
    \ll&  \max_{\substack{T \ll rs U \\ N > N_T}} \frac{X(\log X)^2}{rs T^{3/2} (N/N_T)^2}\max_{\substack{r_0|r,s_0|s \\ T/U \ll r_0 s_0  \ll T }}  r_0s_0 \\
   & \times  \sum_{\substack{w \asymp  W \\ w \, \text{square-free} \\ (w,ar_0s_0)=1 \\} } \bigg| \sum_{1 \leq n \leq N} \lambda_n  \Kl(an,r_0s_0 w)  \bigg|.
\end{align*}
Applying Corollary \ref{cor:sqfree} we get
\begin{align*}
     &\sum_{\substack{w \asymp  W \\ w \, \text{square-free} \\ (w,ar_0s_0)=1 \\} } \bigg| \sum_{1 \leq n \leq N_T} \lambda_n  \Kl(an,r_0s_0 w)  \bigg|  \\
     \ll &d(rs)^{O(1)}(\log X)^{O(1)} N^{1/2} \bigg( N^{1/4} ( s_0^{1/8} W^{5/4}  +     r_0^{1/4} W  )  
         +  N^{1/2} W^{1/2} + \frac{N^{1/2} W^{3/4}}{s_0^{1/8}}\bigg).
\end{align*}
Plugging this in, we are left with bounding four terms separately. We first note that the bound is decreasing for $N> N_T$ so we substitute $N=N_T= \delta^{-2}T^2/X$ and recall that $W=T/r_0s_0$.  For the first term, we get (ignoring the factor of $\delta^{-2} (\log X)^{O(1)} d(crs)^{O(1)}$)
\begin{align*}
    &\ll \max_{T \ll rs U}  \frac{X}{rsT^{3/2}  }\max_{\substack{r_0|r,s_0|s \\ T/U \ll  r_0 s_0  \ll T }} r_0s_0   N_T^{3/4}  s_0^{1/8} W^{5/4}  \\
   & \ll  \frac{X^{1/4}}{rs} \max_{T \ll rs U}T^{5/4} \max_{\substack{r_0|r,s_0|s \\ T/U \ll r_0 s_0  \ll T }}   r_0^{-1/4} s_0^{-1/8} \\
   & \ll  \frac{X^{1/4}}{rs} \max_{T \ll rs U}T^{5/4}  \Bigl(1+ \frac{T}{U s}\Bigr)^{-1/8} \Bigl(1 + \frac{T}{U }\Bigr)^{-1/8}      \ll s^{1/8} U^{5/4}  X^{1/4}    
\end{align*}
by using $r_0s_0 > 1 + \frac{T}{U }$ and $r_0 > 1+ \frac{T}{U  s}$. For the second term we have
\begin{align*}
    & \ll \max_{T \ll rs U}  \frac{X}{rsT^{3/2}  }\max_{\substack{r_0|r,s_0|s \\ T/U \ll  r_0 s_0  \ll T }} r_0s_0 N_T^{3/4}   r_0^{1/4} W \\
    & \ll \frac{X^{1/4}}{rs } \max_{T \ll rs U}  T \max_{\substack{r_0|r,s_0|s \\ T/U \ll  r_0 s_0  \ll T }} r_0^{1/4}  \ll r^{1/4} U X^{1/4} .
\end{align*}
For the third, we get
\begin{align*}
  & \ll \max_{T \ll rs U} \frac{X}{rsT^{3/2}  }\max_{\substack{r_0|r,s_0|s \\ T/U \ll  r_0 s_0  \ll T }}  r_0s_0 N_T W^{1/2}\\
  & \ll  \max_{T \ll rs U} \max_{\substack{r_0|r,s_0|s \\ T/U \ll  r_0 s_0  \ll T }} \frac{1}{rs}T (r_0s_0)^{1/2}  \ll (rs)^{1/2} U
\end{align*}
and for the final part
\begin{align*}
    & \ll\max_{T \ll rs U} \frac{X}{rsT^{3/2}  }\max_{\substack{r_0|r,s_0|s \\ T/U \ll r_0 s_0  \ll T }}  r_0s_0 N_T\frac{ W^{3/4}}{s_0^{1/8}}\\
    & \ll \frac{1}{rs} \max_{T \ll rs U} T^{5/4} \max_{\substack{r_0|r,s_0|s \\ T/U \ll r_0 s_0  \ll T }}   r_0^{1/4} s_0^{1/8}
    \ll  r^{1/2} s^{3/8} U^{5/4}.
\end{align*}
Combining the four bounds gives the result.
\qed
\subsection{Proof of Theorem \ref{Thm:divisorAPsquares}}

Theorem \ref{Thm:divisorAPsquares} is proved similarly to Theorem \ref{Thm:divisorAP} but using Corollary \ref{cor:squares} in place of Corollary \ref{cor:sqfree}. A minor complication is that $(t,u^2)$ can have square-free factors. Similarly as above, with $N_T = \delta^{-2}  T^2 /X  $ we have 
\begin{align*}
    \sum_{\substack{u^2 \sim  U \\ u \, \text{square-free} \\ (u,ars)=1 \\} } | \Delta^\psi(X;rsu^2,a)|  
   \ll \sum_{\substack{u^2 \sim  U \\ u \, \text{square-free} \\ (u,ars)=1 \\} }  \max_{\substack{T \ll rs U \\ N > N_T}} \frac{X(\log X)^2}{rsu^2 T^{1/2} (N/N_T)^2}   \sum_{\substack{t|rsu^2 \\t \sim T}}  \bigg| \sum_{1 \leq n \leq N} \lambda_n  \Kl(an,t)  \bigg|.
\end{align*}
 We then write $t=r_0 s_0 w_0 w^2$ where $(t,u^2)=w_0w^2$ with $(w_0,w)=1$ and $w_0$ square-free, so that $u=v w_0 w$ with $w^2 \asymp W := T/(r_0 s_0 w_0)$.  We sum trivially over $v^2 \asymp Ur_0s_0/(T w_0) $  to get
\begin{align*}
      & \ll  \max_{\substack{T \ll rs U \\ N > N_T}} \frac{X(\log X)^2}{rsU^{1/2} T (N/N_T)^2} \max_{\substack{r_0|r,s_0|s  }} \sum_{\substack{w_0 < T/(r_0s_0)  \\  w_0 T/U \ll  r_0 s_0  \ll T \\ w_0 \, \text{square-free}} } \Bigl(\frac{r_0s_0}{ w_0}\Bigr)^{1/2}  \\
    & \hspace{180pt} \times\sum_{\substack{w^2 \asymp  W \\ w \, \text{square-free} \\ (w,ar_0s_0w_0)=1 \\} } \bigg| \sum_{1 \leq n \leq N} \lambda_n  \Kl(an,r_0s_0 w_0 w^2)  \bigg| \\
    &  \ll  \max_{\substack{T \ll rs U \\ N > N_T}} \frac{X(\log X)^3}{rsU^{1/2} T  (N/N_T)^2} \max_{\substack{r_0|r,s_0|s}} \max_{\substack{w_0 < T/(r_0s_0)  \\  w_0 T/U \ll r_0 s_0  \ll T \\ w_0 \, \text{square-free}} } ( r_0s_0w_0)^{1/2}  \\
     & \hspace{180pt}\times\sum_{\substack{w^2 \asymp  W \\ w \, \text{square-free} \\ (w,ar_0s_0w_0)=1 \\} } \bigg| \sum_{1 \leq n \leq N} \lambda_n  \Kl(an, r_0s_0 w_0 w^2)  \bigg|.
\end{align*}
An application of Corollary \ref{cor:squares} with $s=s_0w_0$ gives
\begin{align*}
    &\sum_{\substack{w^2 \asymp  W \\ w \, \text{square-free} \\ (w,ar_0s_0w_0)=1 \\} } \bigg| \sum_{1 \leq n \leq N} \lambda_n  \Kl(an,r_0s_0 w_0 w^2)  \bigg| \\
    &\ll d(rs)^{O(1)}(\log U)^{O(1)} N^{1/2}\bigg( N^{1/4} ( (s_0w_0)^{1/8} W^{3/4}  +     r_0^{1/4} W^{1/2}  )  
         +  N^{1/2} W^{1/4} + \frac{N^{1/2} W^{1/4+o(1)}}{(s_0w_0)^{1/8}}\bigg).
\end{align*}
The result then follows by similar computations as before, noting that the bound is decreasing in $w_0$ since $W=T/( r_0 s_0 w_0) \ll U/w_0^2$. \qed

\section{Proof of Theorem \ref{thm:APexceptional}} \label{sec:exceptionalproof}

We may assume that $\eps \in ( \frac{\log \log X}{ \log X},\frac{1}{1000})$  since otherwise the statement is trivial by increasing the implied constants. We define $P_1 < Q_1 < P_2 < Q_2$ via
\begin{align} \label{eq:PQdef}
    (P_1,Q_1,P_2,Q_2) = (X^{10 \eps}, X^{\frac{1}{2} \sqrt{\eps}}, X^{2 \sqrt{\eps}}, X^{1/12})
\end{align}
and define the set of good moduli
\begin{align} \label{eq:exceptionalset}
    \mathcal{Q} = \Bigl\{ p_1p_2 k \sim Q: p_1 \in (P_1,Q_1], p_2 \in (P_2,Q_2], \gcd(k,p_1p_2)=1 ] \Bigr\}.
\end{align}
We then have the following simple sieve upper bound for the exceptional set of moduli
\begin{align} \label{eq:sieve}
 \# \{q \sim Q: q \not \in \mathcal{Q}\}\ll Q\Bigl( \frac{\log P_1}{\log Q_1} + \frac{\log P_2}{\log Q_2}+\frac{1}{P_1}+\frac{1}{P_2}\Bigr) \ll \sqrt{\eps} Q.
\end{align}
For the good moduli we have
\begin{align*}
   \sum_{\substack{q \in \mathcal{Q}} }  |   \Delta(X;q,a)|   \leq     \sum_{\substack{p_2 \in (P_2,Q_2] }} \sum_{\substack{Q/(Q_1 p_2)\leq k \leq Q/(P_1 p_2) \\ p_2\nmid k} } \sum_{\substack{p_1 \sim Q/(k p_2)\\ p_1 \nmid k}}   |   \Delta(X;k p_1p_2,a) |.
\end{align*}
By a dyadic partition, it suffices to bound the analogous sum with $X < n  \leq 2X$ in place of $n \leq X$. We insert a smooth approximation $\psi(n/X)$ with $\delta=X^{-\eps}$ where the error term  coming from the edges is negligible by Shiu's bound \cite[Theorem 1]{Shiu}. We write $k=k_1 k_2$ with $k_1$ is cube-free and $k_2$ is cube-full (even a square-free split would suffice). An application of Theorem \ref{Thm:divisorAP} with $(r,s,u) = (k_2p_2,k_1,p_1)$ bounds the contribution from $q \in \mathcal{Q}$ by
\begin{align*}
\ll  & \delta^{-2}(\log X)^{O(1)}  \sum_{\substack{p_2 \in (P_2,Q_2] }} \sum_{\substack{Q/(Q_1 p_2)\leq k_1k_2 \leq Q/(P_1 p_2) \\ p|k_2 \Rightarrow p^3 |k_2 }} d(k_1 k_2)^{O(1)}  \\ & \times\Bigl(    k_1^{1/8} \bigl(\frac{Q}{k_1k_2 p_2}\bigr)^{5/4}X^{1/4}+
 (k_2p_2)^{1/4}  \frac{Q}{k_1 k_2 p_2}  X^{1/4} + ( k_1 k_2 p_2)^{1/2} \frac{Q}{k_1 k_2 p_2} + (k_2 p_2)^{1/2} k_1^{3/8}  \bigl(\frac{Q}{k_1k_2 p_2}\bigr)^{5/4} \Bigr) \\
  \ll & \delta^{-2}(\log X)^{O(1)} \Bigl(X^{1/4} Q^{9/8} Q_1^{1/8} P_2^{-1/8}+ X^{1/4} Q  Q_2^{1/4} + Q^{3/2}  P_1^{-1/2}  + Q^{11/8}P_1^{-1/8} Q_2^{1/8}\Bigr) \\
  \ll &   (\log X)^{O(1)} X^{1+ \eps 25/8   -\sqrt{\eps} 3/16 } + X^{11/12 + 3\eps + 1/48} + X^{1 +\eps 7/2 - 5 \eps } + X^{11/12 + \eps 27/8 - \eps 5/4 + 1/96}  \\
  \ll& (\log X)^{O(1)} X^{1-\eps} \quad \text{ for } \quad \eps < 1/1000.
\end{align*}
 \qed

\section{Proof of Theorem \ref{thm:binarycubic}} \label{sec:binarycubicproof}
We take  $\eps = A \frac{\log\log X}{ \log X}$ for some large $A>0$, let $P_1,Q_1,P_2,Q_2$ be as in \eqref{eq:PQdef}, and set
\begin{align*}
    \mathcal{N} :=  \{ p_1p_2 k \leq X: p_1 \in (P_1^{1/2},Q_1^{1/2}], p_2 \in (P_2^{1/2},Q_2^{1/2}], \gcd(k,p_1p_2)=1 ] \}.
\end{align*}
Then by Shiu's bound \cite[Theorem 1]{Shiu} followed by the sieve bound \eqref{eq:sieve} we have
\begin{align*}
\sum_{ \substack{n_2 \leq X \\ n_2 \not \in \mathcal{N}} }   \sum_{n_1 \leq X} d(n_1 n_2^2 +1)  \ll \Bigl(\frac{\log \log X}{ \log X} \Bigr)^{1/2} X^2 \log X.
\end{align*}
For $n_2 \in \mathcal{N}$ we argue similarly as Section \ref{sec:exceptionalproof} but using Theorem \ref{Thm:divisorAPsquares} to get, once $A$ is sufficiently large,
\begin{align*}
    \sum_{ \substack{n_2 \leq X \\ n_2 \in \mathcal{N}} }   \bigg|\sum_{n_1 \leq X} d(n_1 n_2^2 +1) - \sum_{m \leq n_2^2 X} d(m) \frac{\mathbf{1}_{(m,n_2)=1}}{ \varphi(n_2^2)} \bigg| \ll \frac{X^2}{ (\log X)^A}.
\end{align*}
It then remains to evaluate the main term, where we can restore the exceptional moduli $n_2 \not \in \mathcal{N}$ by another application of Shiu's bound \cite[Theorem 1]{Shiu} and \eqref{eq:sieve}. We get by Perron's formula and shifting the contour
\begin{align*}
        \sum_{ \substack{n_2 \leq X} } \frac{1}{ \varphi(n_2^2)} \sum_{m \leq n_2^2 X} d(m) \mathbf{1}_{(m,n_2)=1}  
      &= X \sum_{ \substack{n_2 \leq X} } \frac{\varphi(n_2) \log (X n_2^2) }{n_2}  + O(X^{2})  =\frac{3}{\zeta(2)} X^2 \log X + O(X^{2}).
\end{align*}
\qed 
\bibliography{divreferences}
\bibliographystyle{abbrv}

\end{document}